\documentclass[reqno,12pt]{amsart}

\usepackage[margin=1in]{geometry}
\usepackage[T1]{fontenc}
\usepackage[utf8]{inputenc}

\usepackage{amsthm, amsmath, amssymb,xcolor,mathtools,booktabs}

\usepackage{tikz}
\usetikzlibrary{calc,shapes}
\usetikzlibrary{decorations.pathmorphing}

\usepackage{caption} 

\usepackage[textsize=small,backgroundcolor=orange!20]{todonotes}

\usepackage[hidelinks]{hyperref}
\usepackage[shortlabels]{enumitem}

\usepackage[noabbrev,capitalize]{cleveref}
\crefformat{equation}{(#2#1#3)}
\crefmultiformat{equation}{(#2#1#3)}{ and~(#2#1#3)}{, (#2#1#3)}{ and~(#2#1#3)}

\usepackage[color,final]{showkeys} 

\colorlet{refkey}{orange!20}
\colorlet{labelkey}{blue!60}

\newtheorem{theorem}{Theorem}

\newtheorem{lemma}[theorem]{Lemma}

\theoremstyle{definition}
\newtheorem{definition}[theorem]{Definition}

\newtheorem{problem}[theorem]{Problem}
\newtheorem{open}[theorem]{Open Problem}

\newtheorem{question}[theorem]{Question}

\theoremstyle{remark}

\newcommand{\abs}[1]{\left\lvert#1\right\rvert}

\newcommand{\ang}[1]{\left\langle #1 \right\rangle}
\newcommand{\angs}[1]{\langle #1 \rangle}
\newcommand{\floor}[1]{\left\lfloor #1 \right\rfloor}
\newcommand{\ceil}[1]{\left\lceil #1 \right\rceil}
\newcommand{\paren}[1]{\left( #1 \right)}

\newcommand{\set}[1]{\left\{ #1 \right\}}

\DeclareMathOperator{\tr}{tr}
\DeclareMathOperator{\rank}{rank}
\DeclareMathOperator{\nullity}{nullity}
\DeclareMathOperator{\mult}{mult}
\DeclareMathOperator{\PSL}{PSL}
\DeclareMathOperator{\Aff}{Aff}

\newcommand{\CC}{\mathbb{C}}

\newcommand{\FF}{\mathbb{F}}
\newcommand{\RR}{\mathbb{R}}

\title{Equiangular lines and eigenvalue multiplicities}
\author{Yufei Zhao}
\date{}

\begin{document}

\maketitle

Sphere packing, spherical codes, and other geometric arrangement problems have captivated the minds of Kepler, Newton, Gauss, and countless others. 
Viazovska's recent breakthrough proof of the optimality of sphere packing arrangements in dimensions 8 and 24 was recognized in her 2022 Fields Medal citation.
Despite the substantial interest and effort, these problems remain poorly understood especially in high dimensions, where precise answers are exceedingly rare.

This article delves into a recent solution of a longstanding problem on optimal configurations of equiangular lines. 
This is a rare instance of a high dimensional packing problem that we now understand quite precisely.
As we embark on our exploration of equiangular lines, we will see a beautiful link to spectral graph theory.
A surprising revelation regarding eigenvalue multiplicities emerges, and it plays a pivotal role in cracking the geometric arrangement problem.

\section*{Equiangular Lines}

How many lines can there be in $\RR^d$, all passing through the origin, and pairwise forming equal angles?
In two dimensions, the answer is three.
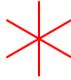
\begin{center}
	\begin{tikzpicture}[scale=.5]
		\draw[red, thick] (90:1) -- (90:-1);
		\draw[red, thick] (30:1) -- (30:-1);
		\draw[red, thick] (-30:1) -- (-30:-1);
	\end{tikzpicture}
\end{center}
In three dimensions, which is already harder to visualize, it turns out that the answer is six, attained by the six main diagonal lines of a regular icosahedron.

Let $N(d)$ denote the maximum number of equiangular lines in $\RR^d$.
Although the determination of each $N(d)$ turns out to be a finite problem, the complexity of the computation grows rapidly. 
The exact value of $N(d)$ is known for only finitely many $d$ (e.g., see \cite{GSY21}).

A short and elegant linear algebraic argument by Gerzon~\cite{LS73}, reproduced in the footnote\footnote{Let $v_1, \dots, v_n \in \RR^d$ be unit vectors, one along each line, with pairwise inner products $\pm \alpha$.
 Let $w_i = v_i \otimes v_i \in \RR^{d^2}$, which can be viewed as a symmetric $d\times d$ matrix, and thus determined by the $\binom{d+1}{2}$ entries on and above the diagonal.
Thus $w_1, \dots, w_n$ lie in a $\binom{d+1}{2}$-dimensional vector space.
Furthermore, $\angs{w_i, w_j} = \angs{v_i, v_j}^2 = \alpha^2$ if $i \ne j$, is equal to $1$ if $i = j$. 
The $n \times n$ matrix with $1$'s on the diagonal and $\alpha^2$ elsewhere has full rank. So $w_1, \dots, w_n$ are linearly independent. 
Therefore $n \le \binom{d+1}{2}$.
}, 
shows that $N(d) \le \binom{d+1}{2}$.
A matching quadratic lower bound $N(d) \ge c d^2$ was proved by de Caen~\cite{deC00}.
Despite further progress, there remains a constant factor gap between the current lower and upper bounds.

These constructions of quadratically many equiangular lines turn out to have pairwise angles approach $90^\circ$ as the dimension increases.
What happens if we fix the angle and let the dimension grow?

Let
$N_\alpha(d)$ 
be the maximum number of lines in $\RR^d$ with pairwise angle $\cos^{-1}\alpha$.
The case of fixed $\alpha > 0$ and large $d$ is particularly interesting.
It turns out that the number of lines grows linearly in $d$.
More precisely,
$N_\alpha(d) \le C_\alpha d$ for some constant $C_\alpha$ \cite{Buk16},
in contrast to $N(d) = \Theta(d^2)$.

\begin{question}
	For fixed $0 < \alpha < 1$, determine $\lim_{d \to \infty} N_\alpha(d) / d$.
\end{question}

With Zilin Jiang, Jonathan Tidor, Yuan Yao, and Shengtong Zhang (students and a postdoc at MIT at the time of the work),
we gave a complete answer to this question~\cite{JTYZZ21}. We will state the main results shortly.
A key step in our solution is a new result in spectral graph theory concerning eigenvalue multiplicities.
In this article, when we talk about the \emph{eigenvalues} of a graph $G$, we are always referring to eigenvalues of the \emph{adjacency matrix} $A_G$.

\begin{theorem} \label{thm:mult}
Every connected bounded degree graph has sublinear second eigenvalue multiplicity.	
More precisely, the second largest eigenvalue of the adjacency matrix of an $n$-vertex connected graph with maximum degree $\Delta$ has multiplicity $O_\Delta(n/\log\log n)$.
\end{theorem}

While much attention has been given to the spectral gap due to its connections to important properties like expansion and mixing (e.g., see the survey~\cite{HLW06} on expander graphs), little is known about eigenvalue multiplicities. 
\cref{thm:mult} is one of the first general results of this form.
We will explore the eigenvalue multiplicity problem later in this article.
It remains a tantalizing open problem to improve the bound of $O_\Delta(n/\log\log n)$ in \cref{thm:mult}.
The best lower bound we know is on the order of $\sqrt{n/\log n}$.

\subsection*{History}

While the general problem of equiangular lines has appeared in papers at least as early as the 1940's, the fixed angle version was considered in a 1973 paper by Lemmens and Seidel~\cite{LS73}, who determined the value of $N_{1/3}(d)$ for all $d$. In particular, they showed that
\[
N_{1/3}(d) = 2(d-1) \quad \text{for all } d \ge 15.
\]
The next case was solved in a 1989 paper by Neumaier~\cite{Neu89}, who showed that
\[
N_{1/5} (d) = \floor{\frac{3}{2} (d-1)} \quad \text{for all sufficiently large $d$.}
\]
As for $N_{1/7}(d)$, Neumaier wrote in his paper that ``the next interesting case will require substantially stronger techniques.''

Progress on this problem stalled until Bukh~\cite{Buk16} reignited interest in the problem in his 2016 paper. He showed that $N_\alpha(d) \le C_\alpha d$ for some constant $C_\alpha$ for each $\alpha$. 
Bukh furthermore conjectured that
$N_{1/(2k-1)}(d) = \frac{k}{k-1} d+ O_k(1)$ for any integer $k\ge 2$.
This was followed by work by Balla, Dräxler, Keevash, and Sudakov~\cite{BDKS18}, who showed that 
\[
\limsup_{d \to \infty} N_\alpha(d)/d \le 1.99 \quad \text{for all }\alpha \ne 1/3.
\]
Jiang and Polyanskii~\cite{JP20} formulated a conjecture for $\limsup_{d \to \infty} N_\alpha(d)/d$ as a function of $\alpha$. We defer the precise statement until \cref{thm:equiang} since it requires a new definition.
They proved their conjecture for all $\alpha> (1 + 2\sqrt{2+\sqrt5})^{-1} = 0.195...$, where the threshold comes from a problem they solved in spectral graph theory. 
They also explained why this is a natural barrier for all existing methods, clarifying Nermaier's earlier comment.
The missing key turned out to be sublinear second eigenvalue multiplicities (\cref{thm:mult}).

\subsection*{Results}

Our work~\cite{JTYZZ21} determined the limit of $N_\alpha(d)/d$ as $d \to \infty$ for all $\alpha$, confirming earlier conjectures of Bukh~\cite{Buk16} and Jiang and Polyanskii~\cite{JP20}.
For many values of $\alpha$, specifically those where the limit is greater than $1$,
 the theorem actually gives the exact value of $N_\alpha(d)$ for large $d$. 
 
 First, let us state the result when $1/\alpha$ is an odd integer, as the answer has a simpler form.

\begin{theorem} \label{thm:equiang-int}
	For every integer $k \ge 2$, there is some $d_0(k)$ so that
	\[
	N_{1/(2k-1)} (d) = \floor{\frac{k}{k-1} (d-1)} \qquad \text{for all } d \ge d_0(k).
	\]
\end{theorem}

We now know that \cref{thm:equiang-int} holds for $d_0(k) = 2^{k^{Ck}}$ for some constant $C$~\cite{Bal21}, whereas conjecturally $d_0(k) = k^C$ is enough.
The bottleneck lies in improving \cref{thm:mult}.

To state the result for all other angles, we need some notation.

\begin{definition}
The \emph{spectral radius order} $k(\lambda)$ of $\lambda > 0$ is defined to be the minimum possible number of vertices of a graph with top eigenvalue $\lambda$.
We set $k(\lambda) = \infty$ if no such graph exists.
\end{definition}

In other words, imagine enumerating all graphs in order of their number of vertices, computing the largest eigenvalue, and halting the first time $\lambda$ comes up. Then $k(\lambda)$ is the number of vertices in the corresponding graph.\footnote{The computation of $k(\lambda)$ given $\lambda$ is an interesting problem that we do not fully understand.
For example, we do not know how to certify that some $\lambda$ has $k(\lambda) = \infty$. 
There are some necessary conditions on $\lambda$ in order that $k(\lambda) < \infty$.
For example $\lambda$ should be a Perron number, namely a positive real algebraic integer whose Galois conjugates are all real and less than $\lambda$ in absolute value.
However, being Perron is not sufficient for $k(\lambda) < \infty$ \cite{Sch23}.}
Now we are ready to state the main result from \cite{JTYZZ21}. See \cref{tab:spec-rad-order} for examples.

\begin{table}[htb]
	\centering	
	\caption{Examples for \cref{thm:equiang}.
	Here $G$ is a graph on $k = k(\lambda)$ vertices with top eigenvalue $\lambda$.\label{tab:spec-rad-order}}

	\renewcommand{\arraystretch}{1.25}
	\begin{tabular}{ccccc}
		\toprule
		$\alpha$ & $\lambda$ & $k$ & $G$ & $N_\alpha(d) \ \forall \text{ large } d$ \\
		\midrule
		$1/3$ & $1$ & $2$ & 
		\begin{tikzpicture}[
		v/.style={circle, fill, minimum size = 2pt,inner sep = 0pt}, 
		baseline={([yshift=-.8ex]current bounding box.center)},
		scale=.4]
			\node[v] (1) at (0,0) {};
			\node[v] (2) at (1,0) {};			
			\draw (1)--(2);
		\end{tikzpicture} 
		& $2(d-1)$
		\\
		$1/5$ & $2$ & $3$ & 
				\begin{tikzpicture}[v/.style={circle, fill, minimum size = 2pt,inner sep = 0pt}, 
		baseline={([yshift=-.8ex]current bounding box.center)},
		scale=.2]
			\node[v] (1) at (90:1) {};
			\node[v] (2) at (210:1) {};
			\node[v] (3) at (-30:1) {};
			\draw (1)--(2)--(3)--(1);
		\end{tikzpicture} 
		& $\floor{\frac{3}{2} (d-1)}$
		 \\
		$1/7$ & $3$ & $4$ &  
		\begin{tikzpicture}[v/.style={circle, fill, minimum size = 2pt,inner sep = 0pt}, 
			baseline={([yshift=-.8ex]current bounding box.center)},
			scale=.4]
			\node[v] (1) at (0,0) {};
			\node[v] (2) at (1,0) {};			
			\node[v] (3) at (1,1) {};
			\node[v] (4) at (0,1) {};			
			\draw (1)--(2)--(3)--(4)--(1)--(3) (2)--(4);
		\end{tikzpicture} 
		& $\floor{\frac{4}{3} (d-1)}$
		\\
		$1/(2k-1)$ & $k-1$ & $k$ & $K_k$ 
		& 
		$\floor{\frac{k}{k-1} (d-1)}$
		\\
		$1/(1+2\sqrt{2})$ & $\sqrt{2}$ & $3$ & 
		\begin{tikzpicture}[v/.style={circle, fill, minimum size = 2pt,inner sep = 0pt}, 
		baseline={([yshift=-.8ex]current bounding box.center)},
		scale=.2]
			\node[v] (1) at (90:1) {};
			\node[v] (2) at (210:1) {};
			\node[v] (3) at (-30:1) {};
			\draw (3)--(1)--(2);
		\end{tikzpicture} 
		& $\floor{\frac{3}{2} (d-1)}$
		 \\				
		\bottomrule
	\end{tabular}
\end{table}

\begin{theorem}
\label{thm:equiang}
	Let $0 < \alpha < 1$. Let $\lambda = (1-\alpha)/(2\alpha)$ and $k = k(\lambda)$ be the spectral radius order.
	\begin{enumerate}[(a)]
		\item If $k < \infty$, then there is some $d_0(\alpha)$ so that 
		\[
		N_\alpha(d) = \floor{\frac{k}{k-1}(d-1)} \qquad \text{for all } d \ge d_0(\alpha).
		\]
		\item If $k = \infty$, then
		\[
		N_\alpha(d) = d+ o(d) \qquad  \text{as } d \to \infty.
		\]
	\end{enumerate}
\end{theorem}

\section*{Connection to Spectral Graph Theory}

The classic \textit{Algebraic Graph Theory} (Graduate Text in Mathematics) by Godsil and Royle dedicates an entire chapter to equiangular lines, which the authors call ``one of the founding problems of algebraic graph theory.''
Let us explain the connection between equiangular lines and spectral graph theory.

Given equiangular lines $\ell_1, \dots, \ell_n$ in $\RR^d$, pick a unit vector $v_i$ along the direction of $\ell_i$ for each $i$. 
There are two choices for $v_i$.
We make the choice arbitrarily for now and revisit this choice later.
Since the lines pairwise  form angle $\theta = \cos^{-1}\alpha$, we have $\ang{v_i, v_j} = \pm \alpha$ for all $i \ne j$.
Define a graph $G$ with vertices $1, \dots, n$, and an edge between $i$ and $j$ if $\ang{v_i, v_j} = -\alpha$. 

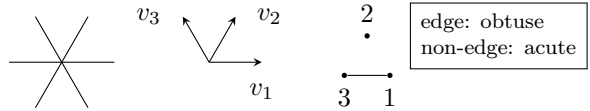
\begin{center}
\begin{tikzpicture}
[v/.style={circle, fill, minimum size = 2pt,inner sep = 0pt}, scale=.7,font=\footnotesize]
\begin{scope}
\draw (0:1)--(0:-1);	
\draw (60:1)--(60:-1);	
\draw (120:1)--(120:-1);	
\end{scope}

\begin{scope}[xshift=2.8cm]
\draw[-stealth] (0,0)--(0:1) node[label=below:$v_1$]{};	
\draw[-stealth] (0,0)--(60:1) node[label=right:$v_2$]{};
\draw[-stealth] (0,0)--(120:1) node[label=left:$v_3$]{};
\end{scope}

\begin{scope}[xshift=5.8cm]
\node[v,label=below:$1$] (1) at (-30:.5) {};
\node[v,label=above:$2$] (2) at (90:.5) {};
\node[v,label=below:$3$] (3) at (210:.5) {};
\draw(1)--(3);
\end{scope}

\node[rectangle,
draw, align=left,
font=\tiny] at (8.3,0.5) 
{
edge: obtuse\\
non-edge: acute
};
\end{tikzpicture}
\end{center}

The Gram matrix has ones on the diagonal and $\pm \alpha$ off diagonal, with signs determined by the graph $G$:
(here $I$ is the identity matrix and $J$ is the all-ones matrix)
\begin{align*}
\text{Gram matrix} 
&=
\begin{pmatrix}
\ang{v_1, v_1} & \cdots & \ang{v_1, v_n} \\
\vdots & \ddots & \vdots \\
\ang{v_n, v_1} & \cdots & \ang{v_n, v_n} \\
\end{pmatrix}
\\
&= 
(1 - \alpha) I - 2\alpha A_G + \alpha J.
\end{align*}

The Gram matrix has rank at most $d$, since $v_1, \dots, v_n \in \RR^d$. Furthermore, the Gram matrix is positive semidefinite.
Conversely, any positive semidefinite $n \times n$ matrix with rank at most $d$ and ones on the diagonal is the Gram matrix of some collection of unit vectors in $\RR^d$. Indeed, we can recover $v_1, \dots, v_n$ up to isometry from their Gram matrix using the Cholesky decompostion.

Via the above correspondence, for a given $\alpha$ and $d$, 
the following two quantities are both equal to $N_\alpha(d)$:
\begin{enumerate}[(a)]
	\item the maximum number of lines in $\RR^d$ with pairwise angle $\cos^{-1}\alpha$;
	\item the maximum number of vertices of a graph $G$ for which
	\[(1-\alpha) I - 2\alpha A_G + \alpha J\]
	 is positive semidefinite and has rank at most $d$.
\end{enumerate}

\section*{Proof Outline}

\subsection*{Lower bound constructions for $N_\alpha(d)$}

As an example, let us show 
\[
N_{1/5}(2\ell + h + 1) \ge 3\ell + h \qquad \text{for integers }\ell, h \ge 0.
\]
We need to exhibit $3\ell + h$ unit vectors in $\RR^{2\ell + h + 1}$ with pairwise inner products $\pm 1/5$.
Rather than providing the coordinates of the vectors, which would be quite unwieldy, it suffices to provide their Gram matrix.
The associated graph $G$ turns out to be a disjoint union of $\ell$ triangles along with $h$ isolated vertices.
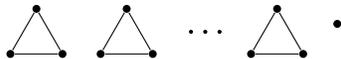
\begin{center}
\begin{tikzpicture}[v/.style={circle, fill, minimum size = 3pt,inner sep = 0pt}, scale=.4]
\foreach \i in {0, 3, 8}{
	\begin{scope}[xshift={\i cm}]
		\node[v] (a) at (90:1) {};
		\node[v] (b) at (210:1) {};
		\node[v] (c) at (-30:1) {};
		\draw (a)--(b)--(c)--(a);
	\end{scope}
	\node at (5.7,0.2) {$\cdots$};
}
	\node[v] at (10,0.5) {};
\end{tikzpicture}
\end{center}
It remains to verify that the associated matrix
\[
(1 - \alpha) I - 2\alpha A_G + \alpha J = 
(4I - 2A_G + J)/5
\]
is positive semidefinite and has rank at most $2\ell + h + 1$.
Indeed, positive semidefiniteness follows from being a positive linear combination of
\begin{itemize}
\item $2I - A_G$, which is positive semidefinite since the top eigenvalue of a triangle is $2$, and
\item $J$, which is also positive semidefinite.
\end{itemize}
As for rank, each of the $\ell$ triangle components of $2I - A_G$ contributes a new dimension to its kernel.
So the dimension of the kernel is $\nullity(2I - A_G) = \ell$.
Thus $\rank(2I - A_G) = 2\ell + h$. 
Since $\rank J = 1$, we have $\rank(2I - A_G + J) \le 2\ell + h +1$. This finishes the proof of $N_{1/5}(2\ell + h + 1) \ge 3\ell + h$.

The same argument shows the general lower bound $N_\alpha(d) \ge \floor{\frac{k}{k-1}(d-1)}$ from \cref{thm:equiang}.

\subsection*{Upper bound on $N_\alpha(d)$}
Applying the rank--nullity theorem, we have
\[
n = \rank(\text{Gram matrix}) + \nullity(\text{Gram matrix}).
\]
The rank of the Gram matrix is at most $d$ since the vectors live in $\RR^d$. 
Since $\rank J = 1$,
\begin{align*}
\MoveEqLeft \mathrm{nullity}(\text{Gram matrix}) 
\\
&= \mathrm{nullity}((1-\alpha)I - 2\alpha A_G + \alpha J)
\\
&\le \mathrm{nullity}((1-\alpha)I - 2\alpha A_G) + 1
\\
&= \mathrm{mult}(\lambda, G) + 1,
\end{align*}
where $\mathrm{mult}(\lambda, G)$ denotes the multiplicity of $\lambda = (1-\alpha)/(2\alpha)$ as an eigenvalue of $A_G$.
Already we see the connection between equiangular lines and eigenvalue multiplicities.

The graph $G$ cannot have more than one eigenvalue greater than $\lambda$. 
Indeed, otherwise their span would contain a unit vector $w$ satisfying $\ang{w, A_Gw} > \lambda$ and $\ang{w, \mathbf 1} = 0$. It would contradict the positive semidefiniteness of the Gram matrix as
\begin{align*}
\MoveEqLeft
\ang{w, (\text{Gram matrix})w} 
\\
&=
\ang{w, ((1-\alpha) I - 2\alpha A_G + \alpha J) w}
\\
&
< 1-\alpha - 2\alpha \lambda 
< 0.
\end{align*}

Thus, for $\lambda$ to be an eigenvalue of $G$, we must be in one of the following two situations:
\begin{enumerate}[(a)]
	\item $\lambda$ is the largest eigenvalue of $G$. In the more interesting case of $k = k(\lambda) < \infty$, this case turns out to produce the equality case when the dimension $d$ is large. The optimal $G$ giving the maximal value of $\mult(\lambda,G)$ ends up being a disjoint union of small graphs each with top eigenvalue at most $\lambda$. This matches the lower bound construction from earlier, giving $n = \floor{k(d-1)/(k-1)}$. This case is easy to analyze. We skip the details and refer the readers to \cite{JTYZZ21}.
	\item $\lambda$ is the second largest eigenvalue of $G$. 
	This is the harder and more interesting case.
	The crux of the argument is to show that this case is suboptimal when the dimension $d$ is large.
	Note that this situation could indeed occur for optimal configurations in small dimensions, such as constructions giving $N(d) = \Theta(d^2)$.
\end{enumerate}

We want to show that the second largest eigenvalue of $A_G$ has small multiplicity. 
By a short argument that we omit, we can assume that $G$ is connected.
Not all connected graphs have small second eigenvalue multiplicity.
For example, a clique on $n$ vertices has top eigevalue $n-1$ with multiplicity $1$, and all $n-1$ remaining eigenvalues equal to $-1$.

Balla et al.~\cite{BDKS18} observed some important constraints on $G$.
Recall that in constructing $G$ from a set of equiangular lines, we took one unit vector $v$ along each line.
In doing so, we made an arbitrary choice whether to take $v$ or $-v$. 
Recall that the edges of $G$ correspond to inner product $-\alpha$ while non-edges correspond to inner product $\alpha$.
Switching some vector to its negative changes the graph by swapping all the edges and non-edges out of the corresponding vertex.

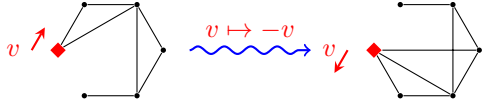
\begin{center}
\begin{tikzpicture}
[v/.style={circle, fill, minimum size = 2pt,inner sep = 0pt}, 
d/.style={diamond, fill, red, minimum size = 6pt,inner sep = 0pt}, 
scale=.7,font=\footnotesize]
\begin{scope}[xshift=0cm]
	\node[d] (1) at (180:1) {};
	\foreach \i in {2,...,6}
		\node[v] (\i) at (240-60*\i:1) {};
	\draw (1)--(2) (1)--(3) (2)--(3) (3)--(4) (4)--(5)--(6) (3)--(5);
	\draw[-stealth,red,thick] (-1.5,0) node[left]{$v$} -- ++(60:.5);
\end{scope}
\draw[decorate, decoration={snake, amplitude=0.5mm},
		->,blue,thick] (1.5,0)-- node[above,red]{$v \mapsto -v$} (3.8,0);

\begin{scope}[xshift=6cm]
	\node[d] (1) at (180:1) {};
	\foreach \i in {2,...,6}
		\node[v] (\i) at (240-60*\i:1) {};
	\draw (1)--(4)(1)--(5)(1)--(6) (2)--(3) (3)--(4) (4)--(5)--(6) (3)--(5);
	\draw[-stealth,red,thick] (-1.5,0) node[left]{$v$} -- ++(60:-.5);
\end{scope}
\end{tikzpicture}
\end{center}
It turns out that one can always switch the graph to get bounded maximum degree.
Here is a precise statement taken from \cite{JTYZZ21}, based on \cite{BDKS18}.
We will sketch a proof in the next section.

\begin{theorem}[Switching to bounded degree] \label{thm:bdd}
	For every $\alpha$, there exists $\Delta$ that depends on $\alpha$ only, so that
	given any configuration of lines (in any dimension) with pairwise angle $\cos^{-1}\alpha$, we can choose one unit vector along each line so that each of these unit vectors has inner product $-\alpha$ with at most $\Delta$ other unit vectors.
\end{theorem}

To complete the proof of \cref{thm:equiang} on the maximum number of equiangular lines with a fixed angle, 
recall we are left with the case when $\lambda$ is the second largest eigenvalue of a connected graph $G$.
By \cref{thm:bdd}, we can assume that $G$ has bounded degree.
By \cref{thm:mult}, we have $\mathrm{mult}(\lambda, G) = o(n)$.
So $n \le d + \mathrm{mult}(\lambda, G) + 1 = d + o(n)$, and thus $n \le d + o(d)$.
When $k = k(\lambda) < \infty$, having $n \le d + o(d)$ is suboptimal compared to the case of $\lambda$ being the largest eigenvalue of $G$, in which the optimal value is $n = \floor{k(d-1)/(k-1)}$.

\section*{Switching to a Bounded Degree Graph}

Let us sketch the proof of \cref{thm:bdd}.
We start with unit vectors $v_1, \dots, v_n \in \RR^d$ with pairwise inner products $\pm \alpha$.
We form the ``negative graph'' $G$ with vertices $\{1, \dots, n\}$, where $ij$ is an edge if $\ang{v_i, v_j} = -\alpha$.
We want to show that we can negate some subset of vectors to turn $G$ into a bounded degree graph. 
Flipping each $v_i$ to $-v_i$ has the effect of interchanging the neighborhood and non-neighborhood of vertex $i$.

\subsection*{Forbidden induced subgraphs}
We start by identifying two types of forbidden induced subgraphs of $G$, using the positive-semidefiniteness of $G$:
\begin{enumerate}
	\item[(I)] Large clique;
	\item[(II)] Some vertex $u$, two large subsets $A$ and $B$ with no edges between them, such that $u$ is complete to $A$ and empty to $B$ (i.e., no edges between $u$ to $B$).
\end{enumerate}
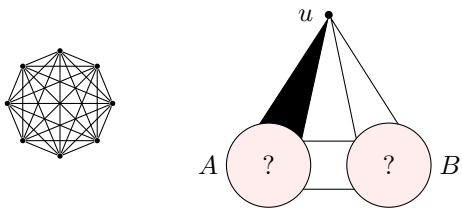
\begin{center}
\begin{tikzpicture}[v/.style={circle, fill, minimum size = 2pt,inner sep = 0pt}, scale=.7,
font=\footnotesize,
baseline={([yshift=-.5ex]current bounding box.center)}]
\foreach \i in {0,...,7}
	\node[v] (\i) at (45*\i:1) {};
\foreach \i in {0,...,7}{
	\foreach \j in {0,...,7}{
            \ifnum\i<\j\relax
                \draw (\i) -- (\j);
            \fi
    }
}
\end{tikzpicture}
\qquad 
\begin{tikzpicture}[v/.style={circle, fill, minimum size = 3pt,inner sep = 0pt}, scale=.4,
font=\footnotesize,
baseline={([yshift=-.5ex]current bounding box.center)}]
\node[v,label=left:$u$] (u) at (0,5) {};

\coordinate (a) at (-2,0);
\begin{scope}[shift=(a)]
\coordinate (a1) at (-40:1);	
\coordinate (a2) at (-40:-1);	
\coordinate (a3) at (90:.8);	
\coordinate (a4) at (90:-.8);	
\end{scope}

\coordinate (b) at (2,0);
\begin{scope}[shift=(b)]
\coordinate (b1) at (40:1);	
\coordinate (b2) at (40:-1);	
\coordinate (b3) at (90:.8);	
\coordinate (b4) at (90:-.8);	
\end{scope}

\draw[fill=black] (u) -- (a1)--(a2) --cycle;
\draw (u) -- (b1)--(b2) --cycle;

\draw (b3)--(b4)--(a4)--(a3)--cycle;

\draw[fill=pink!30] (a) circle (1.4) node[label={[label distance=3mm]left:$A$}]{$?$};
\draw[fill=pink!30] (b) circle (1.4) node[label={[label distance=3mm]right:$B$}]{$?$};
\end{tikzpicture}

\end{center}

Here ``large'' means a large constant.
We will be rather imprecise with constants and quantifiers in this proof sketch, so as to not be bogged down by details.

Let us first address (I).
The Gram matrix restricted to an $\ell$-vertex clique is
\[
\begin{pmatrix}
	1 & -\alpha & \cdots & -\alpha \\
	-\alpha & 1 & \cdots & -\alpha \\
	\cdots & \cdots & \ddots & \vdots \\
	-\alpha & -\alpha & \cdots & 1 
\end{pmatrix}.
\]
The sum of entries is $\ell - \ell(\ell-1)\alpha$, which must be nonnegative due to the matrix being positive semidefinite.
So $\ell \le 1 + 1/\alpha$. 
In other words, $G$ has no clique on more than $1 + 1/\alpha$ vertices.

The proof of (II) is based on a similar idea except that we need a slightly more complicated test vector for positive semidefiniteness. 
The Gram matrix restricted to $\set{u} \cup A \cup B$ is
\[
\paren{\footnotesize 
	\begin{array}{c|ccc|ccc}
		1 & -\alpha &  \cdots & -\alpha & \alpha & \cdots & \alpha  \\
		\hline
		-\alpha & 1  &   & ?  & 1  &   & \alpha \\		
		\vdots  &   & \ddots &   &  &  \ddots  & \\		
		-\alpha & ?  &   &  1 & \alpha  &   & 1 \\
		\hline		
		\alpha & 1  &   & \alpha   & 1  &   & ? \\		
		\vdots &  & \ddots &  &  & \ddots &\\
		\alpha & \alpha  &   & 1  & ? &   & 1 \\
	\end{array}
}
\]
Our test vector $v$ assigns a small constant $\epsilon > 0$ to $u$, as well as $1/\abs{A}$ to each vertex of $A$, and likewise $1/\abs{B}$ to each vertex of $B$. Then
\begin{align*}
v^\intercal M v 
&\approx 
\begin{pmatrix}
	\epsilon \\ 1 \\ -1 \\
\end{pmatrix}^\intercal
\begin{pmatrix}
	1 & -\alpha & \alpha \\
	-\alpha & \star & \alpha \\
	\alpha & \alpha & \star
\end{pmatrix}
\begin{pmatrix}
	\epsilon \\ 1 \\ -1 \\
\end{pmatrix}
\\
&= 2(\star - \alpha) - 2\epsilon \alpha  + \epsilon^2 < 0.
\end{align*}
Here $\star \in [-\alpha,\alpha]$.
The first approximation ignores the negligible contributions from diagonal matrix entries, which is valid if $A$ and $B$ are large compared to $1/\epsilon$.

\subsection*{Switching argument}
Using the above forbidden induced subgraphs, let us now argue that $G$ can be switched to a bounded degree graph.

By Ramsey's theorem\footnote{Balla~\cite{Bal21} later found another proof not using Ramsey's theorem, leading to a better quantitative result. We now know that, in \cref{thm:bdd}, one can take $\Delta$ to be $O(\alpha^{-4})$.}, not having a large clique implies that we must have a large independent set $S$, whose size grows with $n$.

Consider any vertex $u$ outside $S$. We claim that $u$ has either a bounded number of neighbors in $S$ or a bounded number of non-neighbors in $S$.
Indeed, let $A$ and $B$ be respectively the neighbors and non-neighbors of $u$ in $S$. 
There are no edges between $A$ and $B$ since they are both contained in the independent set $S$.
If $A$ and $B$ were both large, then together with $u$ they would form a type II forbidden configuration.

Thus, for every vertex $u$ outside $S$, by negating the vector corresponding to $u$ if necessary, we may assume that $u$ has only a bounded number of neighbors in $S$.

Finally, we claim that the resulting graph has bounded degree. 
Suppose some vertex $u$ has many neighbors.
Let $A$ be a large but bounded subset of neighbors of $u$.
Recall that each vertex in $A$ has a bounded number of neighbors in $S$. 
Let $B$ be $S$ minus the neighborhood of $\set{u}\cup A$. 
Then $u, A, B$ would once again violate (II).

\section*{Sublinear Second Eigenvalue Multiplicity}

Let us now discuss \cref{thm:mult}.
It says that a connected $n$-vertex graph with maximum degree\footnote{We view $\Delta$ as a constant, though the dependence on $\Delta$ is quite mild as the hidden leading constant is $O(\log \Delta)$.}
$\Delta$ has second eigenvalue multiplicity at most $O_\Delta(n / \log\log n)$.
Before diving into the proof, let us discuss some near miss examples that illustrate the necessity of the hypotheses.

\emph{Bounded degree.} A complete graph on $n$ vertices has second eigenvalue multiplicity $n-1$. Other examples with linear eigenvalue multiplicity include strongly regular graphs, such as Paley graphs.

\emph{Connectedness.} A disjoint union of $n/2$ edges has top eigenvalue $1$ repeated $n/2$ times. So the theorem is false without the connectivity hypothesis.

\emph{Second eigenvalue.} The following graph has eigenvalue 0 with linear multiplicity. One such eigenvector is shown (unlabeled vertices are set to zero).
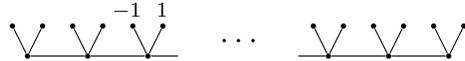
\begin{center}
\begin{tikzpicture}[v/.style={circle, fill, minimum size = 2pt,inner sep = 0pt},scale=0.4]
  \foreach \i in {0,1,2,5,6,7} {
      \begin{scope}[xshift=\i*2cm]
        \node[v] (\i) at (0,0) {};
    	\node[v] (\i0) at (-.5,1) {};
    	\node[v] (\i1) at ( .5,1) {};
		\draw (\i)--(\i0);
		\draw (\i)--(\i1);
	\end{scope}
  }
  \node[font=\tiny,inner sep=0pt, baseline=0pt] at (  3.3,1.5) {$-1$};
  \node[font=\tiny,inner sep=0pt, baseline=0pt] at (4.2,1.5) {$\phantom{+}1$};
  \draw (0)--(1)--(2) (5)--(6)--(7);
  \draw (2) -- ++(1,0);
  \draw (5) -- ++(-1,0);
  \node at (7.1,.5) {$\cdots$};
\end{tikzpicture}	
\end{center}
Furthermore, here is a graph whose least eigenvalue $-3$ has linear multiplicity, arising from the eigenvector of $K_{3,3}$ with eigenvalue $-3$.
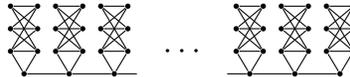
\begin{center}
\begin{tikzpicture}[v/.style={circle, fill, minimum size = 2pt,inner sep = 0pt},scale=0.3]
  \foreach \i in {0,1,2,5,6,7} {
      \begin{scope}[xshift=\i*2cm]
        \node[v] (\i) at (0,0) {};
        \foreach \j in {1,2,3}{
        	\node[v] (\i\j0) at (-.6,\j) {};
        	\node[v] (\i\j1) at ( .6,\j) {};
        }
        \foreach \j in {1,2,3}
        	\foreach \k in {1,2,3}
        		\draw (\i\j0)--(\i\k1);
		\draw (\i)--(\i10);
		\draw (\i)--(\i11);
	\end{scope}
  }
  \draw (0)--(1)--(2) (5)--(6)--(7);
  \draw (2) -- ++(1,0);
  \draw (5) -- ++(-1,0);
  \node at (7.1,1) {$\cdots$};
\end{tikzpicture}	
\end{center}
So \cref{thm:mult} is really about the top of the spectrum.
Essentially the same proof shows that for any constant $j$, the $j$-th largest eigenvalue of a connected graph with maximum degree $\Delta$ has multiplicity $O_{\Delta, j}(n/\log\log n)$.

\subsection*{Proof of sublinear second eigenvalue multiplicity}

Let us prove that $\mult(\lambda, G) = O_\Delta(n/\log\log n)$ for a connected $n$-vertex graph $G$ with maximum degree $\Delta$ and second largest eigenvalue $\lambda$.
We use the trace method to bound eigenvalue multiplicities. 
Write $\lambda_1(G) \ge \cdots \ge \lambda_n(G)$ for the eigenvalues of $A_G$, including multiplicities.
We have
\[
\mult(\lambda, G) \lambda^{2s} \le \sum_{i=1}^n \lambda_i(G)^{2s} = \tr A_G^{2s}.
\]
The first inequality is quite wasteful and so the inequality on its own is not enough to give a sublinear bound on $\mult(\lambda, G)$.
The right-hand side counts closed walks of length $2s$ in $G$.
Such a walk must stay in the radius $s$ neighborhood of its starting vertex. 
This leads us to examine the spectral radius of balls in $G$, which we refer to informally as a ``local spectral radius.'' See \cref{fig:local}.

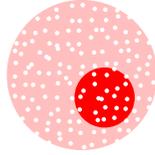
\begin{figure}
\centering
\begin{tikzpicture}
\fill[pink] (0,0) circle (1);
\fill[red] (0.3,-.3) circle (.4);
\foreach \x/\y in {
-0.89/-0.28, -0.91/-0.14, -0.89/-0.02, -0.88/0.16, -0.89/0.32, -0.87/0.42, -0.79/-0.58, -0.72/-0.46, -0.78/-0.3, -0.75/-0.13, -0.75/0.03, -0.71/0.1, -0.81/0.28, -0.76/0.41, -0.77/0.57, -0.6/-0.78, -0.61/-0.6, -0.58/-0.44, -0.66/-0.33, -0.65/-0.13, -0.57/0.03, -0.67/0.14, -0.56/0.3, -0.66/0.45, -0.62/0.56, -0.59/0.75, -0.44/-0.9, -0.46/-0.73, -0.45/-0.64, -0.41/-0.52, -0.49/-0.29, -0.52/-0.15, -0.44/0.02, -0.44/0.22, -0.47/0.24, -0.43/0.42, -0.49/0.59, -0.44/0.75, -0.29/-0.85, -0.3/-0.78, -0.31/-0.61, -0.32/-0.46, -0.33/-0.3, -0.32/-0.11, -0.3/0.03, -0.29/0.14, -0.32/0.35, -0.33/0.39, -0.33/0.56, -0.33/0.77, -0.34/0.87, -0.15/-0.91, -0.15/-0.73, -0.16/-0.62, -0.13/-0.49, -0.13/-0.31, -0.16/-0.11, -0.21/-0.04, -0.17/0.1, -0.13/0.34, -0.18/0.41, -0.18/0.59, -0.13/0.72, -0.13/0.92, 0.03/-0.96, -0.02/-0.77, -0.04/-0.62, -0.03/-0.46, -0.02/-0.29, 0.02/-0.21, -0.01/0.02, 0.01/0.11, -0.01/0.29, -0.04/0.41, -0.04/0.59, -0.03/0.72, 0.01/0.93, 0.13/-0.93, 0.16/-0.76, 0.21/-0.6, 0.18/-0.5, 0.22/-0.31, 0.17/-0.12, 0.16/0.01, 0.17/0.12, 0.16/0.3, 0.12/0.5, 0.19/0.6, 0.13/0.73, 0.17/0.88, 0.3/-0.88, 0.29/-0.75, 0.29/-0.59, 0.31/-0.43, 0.26/-0.32, 0.33/-0.21, 0.33/0.0, 0.33/0.15, 0.27/0.31, 0.34/0.5, 0.27/0.6, 0.32/0.73, 0.31/0.87, 0.47/-0.88, 0.5/-0.73, 0.43/-0.57, 0.46/-0.44, 0.43/-0.28, 0.47/-0.18, 0.43/0.08, 0.45/0.15, 0.52/0.32, 0.45/0.43, 0.44/0.63, 0.48/0.78, 0.47/0.86, 0.62/-0.74, 0.59/-0.61, 0.58/-0.45, 0.6/-0.32, 0.58/-0.12, 0.59/-0.02, 0.62/0.19, 0.62/0.33, 0.6/0.44, 0.58/0.57, 0.63/0.77, 0.74/-0.59, 0.7/-0.48, 0.75/-0.31, 0.78/-0.14, 0.72/-0.02, 0.76/0.08, 0.73/0.25, 0.73/0.46, 0.76/0.56, 0.89/-0.43, 0.9/-0.32, 0.95/-0.1, 0.87/0.02, 0.87/0.14, 0.9/0.24, 0.89/0.43
}
	\fill[white] (\x,\y) circle (.03);
\end{tikzpicture}
\caption{A key idea in the proof of \cref{thm:mult} is that removing a net of vertices (white dots) significantly lowers local spectral radii ($\lambda_1$ of the dark disk). \label{fig:local}}
\end{figure}

A key new idea is to delete a small net from the graph. 
An \emph{$r$-net} in a graph $G$ is a subset of vertices such that every vertex in $G$ lies within distance $r$ from some vertex in the $r$-net. 
We show that \textbf{deleting a net significantly reduces local spectral radii}.
This observation combined with the earlier inequality yields a good upper bound on the multiplicity of $\lambda$ in $G$ with the net removed.
Finally, to bound $\mult(\lambda, G)$, observe that by the Cauchy eigenvalue interlacing theorem, deleting $h$ vertices can reduce each eigenvalue multiplicity by at most $h$.

\begin{lemma}[Finding a small net] \label{lem:net}
	For every $n$ and $r$, every connected $n$-vertex graph has an $r$-net of size at most $\ceil{n/(r+1)}$. 
\end{lemma}

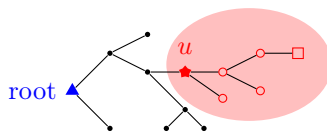
\begin{center}
\begin{tikzpicture}[
		v/.style={circle, fill, minimum size = 2pt,inner sep = 0pt}, 
		b/.style={regular polygon, regular polygon sides = 3,blue,fill, minimum size = 6pt,inner sep = 0pt}, 
		s/.style={star,red,fill, minimum size = 5pt,inner sep = 0pt}, 
		r/.style={circle, red,draw, minimum size = 3pt,inner sep = 0pt}, 
		rr/.style={rectangle,red,draw, minimum size = 4pt,inner sep = 0pt}, 
		baseline={([yshift=-.8ex]current bounding box.center)},
		scale=.5,
		font=\footnotesize
]

\fill[pink] (4.7,.7) circle (2.2 and 1.5);
\node[b,label={[blue]left:root}] (0) at (0,0) {};
\node[v] (00) at (1,1) {};
\node[v] (01) at (1,-1) {};
\node[v] (000) at (2,1.5) {};
\node[v] (001) at (2,0.5) {};
\node[s,label={[red]90:$u$}] (0010) at (3,0.5) {};
\node[v] (0011) at (3,-.5) {};
\node[r] (00101) at (4,-.2) {};
\node[v] (00110) at (2.5,-1) {};
\node[v] (00111) at (3.5,-1) {};
\node[r] (00100) at (4,0.5) {};
\node[r] (001000) at (5,1) {};
\node[r] (001001) at (5,0) {};
\node[rr] (0010000) at (6,1) {};

\draw (0)--(00)--(000) (00)--(001)--(0010)--(00100)--(001000)--(0010000) (0010)--(00101) (00100)--(001001) (001)--(0011)--(00110) (0011)--(00111) (0)--(01);

\end{tikzpicture}
	
\end{center}

\begin{proof}

First replace the graph by a spanning tree. 
	Pick an arbitrary vertex as the root. 
Pick a vertex furthest from the root, and walk $r$ steps towards the root to end up at vertex $u$.
Add $u$ to the net and keep only the the subtree containing the root after deleting $u$.
This removes at least $r+1$ vertices. Repeat.
\end{proof}

\begin{lemma}[Net removal reduces spectral radius] \label{lem:spec-rad-decr}
	If $H$ (with at least one vertex) is obtained from $G$ by deleting an $r$-net, then $\lambda_1(H)^{2r} \le \lambda_1(G)^{2r} - 1$.
\end{lemma}

\begin{proof}
We can assume that $G$ has no isolated vertices.
 Write $A_G$ for the adjacency matrix of $G$, and $A_H$ the matrix obtained by zeroing out the rows and columns corresponding to vertices of the deleted net.
 Clearly $A_H^{2r} \le A_G^{2r}$ entrywise. 
Let $v$ be a vertex of $G$.
Consider the diagonal entries indexed by $(v,v)$ in $A_H^{2r}$ and $A_G^{2r}$.
They count closed walks of length $2r$ starting at $v$.
There are always strictly more such walks in $G$ than in $H$ since at least one such walk passes through the net.
Thus $A_H^{2r} \le A_G^{2r} - I$ entrywise, and the lemma follows.
\end{proof}

Let us write $B_G(v, r)$ for the subgraph of $G$ induced by all vertices within distance $r$ of $v$.

\begin{lemma}[Local vs global spectra] \label{lem:local-global}
	For every graph $G$ and positive integer $s$,
	\[
	\sum_{i=1}^{\abs{G}}\lambda_i(G)^{2s} \le \sum_{v \in V(G)} \lambda_1(B_G(v, s))^{2s}.
	\]
\end{lemma}

\begin{proof}
	The left-hand side, being the trace of $A_G^{2s}$, counts closed walks of length $2s$ in $G$. The number of walks that begin and end at vertex $v$ is $1_v^\intercal A_G^{2s} 1_v$. Since walks must stay within a radius $r$ ball of $v$, this quantity also equals to $1_v^\intercal A_{B_G(v, s)}^{2s} 1_v \le \lambda_1(B_G(v, s))^{2s}$. Finally, sum over all vertices $v$ in $G$.
\end{proof}

Let us now prove $\mult(\lambda, G) = O_\Delta(n/\log\log n)$.
Let $r = \ceil{c \log\log n}$ and $s = \ceil{c \log n}$ for some small constant $c>0$ that is allowed to depend on $\Delta$.

The set of vertices $v$ in $G$ with $\lambda_1(B_G(v, s+1)) > \lambda$ must all lie within distance $2s + 2$ of each other, or else it would violate the Courant--Fischer min-max characterization of the second largest eigenvalue $\lambda$. 
Thus the number of such vertices is at most $\Delta^{2s+2} \le n^{3c}$. Let $H$ denote the graph obtained from $G$ by removing all such vertices, as well as removing an $r$-net of size $\ceil{n/(r+1)}$ in $G$ (which exists by \cref{lem:net}). The number of vertices removed is $O_\Delta(n/\log\log n)$.
We have
\[
\mult(\lambda, H)\lambda^{2s} 
\le \sum_i \lambda_i (H)^{2s} 
\le \sum_{v \in V(H)} \lambda_1 (B_H(v, s))^{2s}
\]
by \cref{lem:local-global}. 
Also, for every $v \in V(H)$, we have $\lambda_1(B_G(v, s+1)) \le \lambda$.
Recall that $H$ is obtained from $G$ by removing an $r$-net.
In $B_G(v,s+1)$, the complement of $B_H(v,s)$ is an $r$-net of $B_G(v,s+1)$, since starting with any vertex in $G$, we can walk $r$ steps to a vertex in the net in $G$, although this may end up taking us outside the radius $s$ ball around $v$.
By \cref{lem:spec-rad-decr},
\[
\lambda_1 (B_H(v, s))^{2s} \le \paren{\lambda^{2r}-1}^{s/r}.
\]
Putting the inequalities together,
\[
\mult(\lambda, H) \le n \paren{1 - \lambda^{-2r}}^{s/r} \le e^{-\sqrt{\log n}} n.
\]
To conclude the proof, we apply the Cauchy eigenvalue interlacing theorem\footnote{The Cauchy eigenvalue interlacing theorem tells us that if we remove the first row and column of a real symmetric matrix, then the resulting eigenvalues $\lambda'_1 \ge \cdots \ge \lambda'_{n-1}$ interlace with the original eigenvalues $\lambda_1 \ge \cdots \ge \lambda_n$ in the sense that $\lambda_1 \ge \lambda'_1 \ge \lambda_2 \ge \lambda'_2 \ge \cdots$. 
In particular if some eigenvalue $\lambda$ has multiplicity $k$ in the larger matrix, then it has multiplicity at least $k-1$ in the smaller matrix.
Thus, for graphs, removing $h$ vertices from $G$ cannot decrease the eigenvalue multiplicity of $\lambda$ by more than $h$.}. 
Since $H$ is obtained from $G$ by removing $O_\Delta(n/\log\log n)$ vertices, we  have
\begin{align*}
\mult(\lambda, G) 
&\le \mult(\lambda, H) + O_\Delta(n/\log\log n)
\\
&= O_\Delta(n/\log\log n).
\end{align*}
This concludes the proof of \cref{thm:mult}.

It remains an intriguing open problem to determine the optimal quantitative bounds in \cref{thm:mult}. 
This could improve the equiangular lines result by reducing $d_0(\alpha)$ in \cref{thm:equiang}(a) and improving the implicit error term in \cref{thm:equiang}(b).

\begin{open}
What is the maximum possible second eigenvalue multiplicity of a connected $n$-vertex graph with maximum degree $\Delta$?
In particular, is it $O_\Delta(n^{1-c})$ for some $c >0$?
\end{open}

The eigenvalue multiplicity bound is central to equiangular lines. It would be exciting to find additional applications.

\subsection*{Lower bound constructions}

How can we construct graphs with high second eigenvalue muliplicity? Eigenvalues tend not to collide by accident (search ``eigenvalue repulsion'' for some neat graphics). 
One method of forcing eigenvalue multiplicity is through symmetry.
Indeed, an eigenvector may be transformed into additional eigenvectors via automorphisms of the graph.

Given a Cayley graph on $n$ vertices, a basic result from group representation theory allows us to decompose $\CC^n$ into eigenspaces according to irreducible representations of the group. 
In particular, each irreducible representation $W$ of the group generates $\dim W$ eigenvalues each repeated $\dim W$ times (there can be further equalities between these eigenvalues). 
The trivial representation gives the top eigenvalue.
If a group has no small non-trivial irreducible representations (such groups are called ``quasirandom'' by Gowers~\cite{Gow08}), then every Cayley graph on this group has large second eigenvalue multiplicity. 

A standard example of a group without small non-trivial representations is $\PSL(2,q)$.
This group has order $n = \Theta(q^3)$. All its non-trivial representations have dimension at least $(q-1)/2 = \Theta(n^{1/3})$. 
This gives examples of connected bounded degree Cayley graphs with second eigenvalue multiplicity at least $\Omega(n^{1/3})$.
Together with Haiman, Schildkraut, and Zhang~\cite{HSZZ22}, we found constructions giving the following.

\begin{theorem}\label{thm:constructions}
\begin{enumerate}
	\item [(a)] There exist infinitely many connected $n$-vertex graphs with maximum degree 4 and second eigenvalue multiplicity at least $\sqrt{n/\log_2 n}$.
	\item [(b)] There exist infinitely many connected $n$-vertex 18-regular Cayley graphs with second eigenvalue multiplicity at least $n^{2/5} - 1$.
\end{enumerate}
\end{theorem}

Here is the construction for \cref{thm:constructions}(a).

\begin{description}
	\item [Step 1]
	Take a 4-regular Cayley graph on the group $\Aff(\FF_p) \cong \FF_p^\times \ltimes \FF_p$ of affine transformations of $\FF_p$, with edges generated by
	\begin{enumerate}
		\item[(i)] a multiplicative shift by a generator of $\FF_p^\times$, and 
		\item[(ii)] an additive shift by $1 \in \FF_p$.
	\end{enumerate}
	\item [Step 2] Subdivide each type (ii) edge into a path of length $\ceil{\log p}$.
\end{description}

Let us provide some insight into why this construction works.
The irreducible representations of $\Aff(\FF_p)$ consist of the one-dimensional characters of $\FF_p^\times$ as well as a $(p-1)$-dimensional representation. The $(p-1)$-dimensional representation arises from restricting the permutation action of $\Aff(\FF_p)$ on $\FF_p$ to the orthogonal complement of the all-ones vector.

Consider the Cayley graph on $\Aff(\FF_p)$ generated only by type (i) edges. This is a disconnected graph with $p$ components. Its top eigenvalue has multiplicity $p$, which comes from the trivial representation (with multiplicity one) and the $(p-1)$-dimensional representation (with multiplicity $p-1$).

Adding type (ii) edges may cause some eigenvalues to be reordered.
Elongating edges to paths in Step 2 prevents this order shuffling. 
In a sense, spectral effects deteriorate over long paths, and subdividing type (ii) edges feels like nearly removing them.

To prove \cref{thm:constructions}(b), 
we used a different group whose Cayley graph came with a spectral gap. 
This renders the subdivision in Step 2 unnecessary.

Constructing graphs with second eigenvalue multiplicity asymptotically greater than $n^{1/2}$ using ideas based on group symmetries seems difficult. This is because all irreducible representations of a group $\Gamma$ have a dimension of at most $\sqrt{\abs{\Gamma}}$ due to $\abs{\Gamma} = \sum_{\rho \text{ irrep}} (\dim \rho)^2$.

In typical constructions, the actual value of the second eigenvalue $\lambda$ fluctuates.
What if we ask $\lambda$ to remain fixed? It was previously conjectured that $\mult(\lambda, G) = O_{\Delta, \lambda}(1)$. However, Schildkraut~\cite{Sch23} recently constructed counterexamples. Using his constructions, he showed that in \cref{thm:equiang}(b), for certain values of $\alpha$ where $k=\infty$, the estimate $N_\alpha(d) = d + O_\lambda(d/\log\log d)$ cannot be sharpened to $d + o(\log\log d)$.

The proof of \cref{thm:mult} is analytic in nature, allowing for some margin of error. Specifically, it shows that there are $O_\Delta(n/\log\log n)$ eigenvalues within the interval $[(1-1/\log n)\lambda_2,\lambda_2]$. 
In \cite{HSZZ22}, we constructed a family of connected bounded degree graphs roughly matching this bound.
This construction involves expanders and long paths. 
It highlights the potential limitations in pushing the proof of \cref{thm:mult} to provide better quantitative bounds.

\subsection*{Related upper bounds}

Among regular graphs, McKenzie, Rasmussen, and Srivastava~\cite{MRS21} improved the second eigenvalue multiplicity bound to $O(n/(\log n)^c)$. More generally, they proved their result for the normalized adjacency matrix $D_G^{-1/2} A_G D_G^{-1/2}$, where $D_G$ is the diagonal matrix of degrees. This matrix corresponds to simple random walks on the graph $G$.
They proved a nice intermediate result about random walks that is interesting in its own right. They showed that a typical closed walk of length $2k$ covers at least $k^c$ vertices. 
To derive the eigenvalue multiplicity bound, they note that removing a random set of vertices (in contrast to removing a net as we did earlier) destroys many closed walks.

An earlier result by Lee and Makarychev~\cite{LM08} tells us that non-expanding Cayley graphs must have small second eigenvalue multiplicity. More precisely, given a connected Cayley graph with doubling constant $K = \max_{r > 0} \abs{B(2r)}/\abs{B(r)}$, where $B(r)$ is the radius $r$ ball, the second eigenvalue multiplicity is $e^{O((\log K)^2)}$.
As a corollary, connected bounded degree Cayley graphs on abelian groups or on nilpotent groups of bounded step always has bounded second eigenvalue multiplicity.
Therefore Cayley graph constructions with large second eigenvalue multiplicity must come from highly nonabelian groups, consistent with examples such as $\PSL(2,q)$. 
Interestingly, Lee and Makarychev's work builds on ideas from Riemannian geometry, namely works by Colding and Minicozzi~\cite{CM97}, as well as Kleiner~\cite{Kle10}. These are all related to Gromov's theorem that groups of polynomial growth are virtually nilpotent.

\section*{Further Questions}

\subsection*{Spherical codes with restricted angles}

A natural extension of equiangular lines involves spherical codes with restricted angles. 
Given $L \subset [-1,1)$, an \emph{$L$-code} is a configuration of points on a unit sphere in $\RR^d$ whose pairwise inner products lie in $L$. Here are some notable examples.
\begin{enumerate}
	\item Equiangular lines: $L = \{-\alpha, \alpha\}$.
	By assigning a unit vector to each line direction, the inner products between these vectors are $\pm \alpha$.
	\item Spherical $t$-distance set: $\abs{L} = t$.
	\item Classical spherical codes: $L = [-1,\cos\theta]$. These are unit vectors pairwise separated by angle at least $\theta$.
	\item Kissing number problem: $L = [-1,\frac12]$. What is the maximum number of nonoverlapping unit balls that all touch a unit ball?
\end{enumerate}

\begin{problem}
Let $N_L(d)$ denote the maximum size of an $L$-code in $\RR^d$. 
Determine, for each fixed $L \subset [-1,1)$, the growth rate of $N_L(d)$ as $d \to \infty$.	
\end{problem}

This general problem is exceptionally challenging.
For example, the best known lower and upper bounds on the classical spherical code problem (i.e., $L = [-1,\cos\theta]$) are exponentially far apart.
The best known upper bounds for classical spherical codes are due to Kabatiansky and Levenshtein~\cite{KL78} from the 1970's, building on Delsarte's influential linear programming bound for error correcting codes.
These bounds also give the state-of-the-art upper bounds on sphere packing densities in high dimensions.

The work described in this article is most relevant to $L \subset [-1, -\beta] \cup \set{\alpha_1, \dots, \alpha_k}$ for some $\alpha_1, \dots, \alpha_k, \beta > 0$, in which case \cite{BDKS18} shows that $N_L(d) = O_{\beta, k}(d^k)$. 
It remains an open problem to pinpoint $N_L(d)$, analogous to equiangular lines.

There was some partial success in the case of spherical two-distance sets, specifically when $L = \set{\alpha, -\beta}$ for some $\alpha, \beta > 0$.
This case is open in general, although we now know how to solve it when $\alpha < 2\beta$ or $(1-\alpha)/(\alpha + \beta) < \lambda^*$ for some special constant $\lambda^* \approx 2.02$ that arises from spectral graph theory for signed graphs~\cite{JTYZZ23,JP21}.
Whereas equiangular lines are connected to eigenvalue multiplicity in graphs, 
there is strong evidence that spherical two-distance sets are linked to eigenvalue multiplicity in signed graphs. 
However, we do not know what is the right analog of the the sublinear second eigenvalue multiplicity for signed graphs. 
The most straightforward statement is false.
We saw earlier a connected bounded degree whose smallest eigenvalue has linear multiplicity.
By putting $-1$ on every edge, this gives a connected signed graph with linear top eigenvalue multiplicity.

\subsection*{Complex equiangular lines}

Another direction concerns complex equiangular lines. 

\begin{problem}
	For a fixed $0 < \alpha < 1$ and large $d$, what is the maximum number of unit vectors in $\CC^d$ whose pairwise Hermitian inner products all have magnitude $\alpha$?
\end{problem}

Recent work by Balla~\cite{Bal21} proved an upper bound of $O_\alpha(d)$. 
A more precise estimate remains elusive.
Furthermore, by viewing complex lines as real two-dimensional flats, we see that there are also natural extensions of the problem to configurations of higher dimensional subspaces.
There is a vast amount to explore in this space.

We finish by mentioning a beautiful conjecture called Zauner's conjecture. 
Recall from the beginning of the article that there are at most $\binom{d+1}{2}$ equiangular lines in $\RR^d$.
Essentially the same proof shows that there are at most $d^2$ equiangular lines in $\CC^d$. 
The $\RR^d$ bound is rarely tight.
On the other hand, Zauner's conjecture says that there exist $d^2$ equiangular lines in $\CC^d$ for all $d$.
This conjecture has been verified in low dimensions.
This problem has interested mathematicians and physicists alike, partly due to connections to quantum physics.
A configuration of $d^2$ equiangular lines in $\CC^d$, also known by the name SIC-POVM, would give a system of quantum measurements with ideal properties.
Recent work has connected this problem to deep ideas in algebraic number theory, namely Stark's conjectures and Hilbert's 12th problem.

\section*{Acknowledgments}
Zhao was supported by NSF CAREER Award DMS2044606 and a Sloan Research Fellowship. 
He thanks his many students and collaborators for their work on this topic.


\newcommand{\etalchar}[1]{$^{#1}$}
\providecommand{\bysame}{\leavevmode\hbox to3em{\hrulefill}\thinspace}
\providecommand{\MR}{\relax\ifhmode\unskip\space\fi MR }
\providecommand{\MRhref}[2]{%
  \href{http://www.ams.org/mathscinet-getitem?mr=#1}{#2}
}
\providecommand{\href}[2]{#2}

\end{document}